\documentclass[11pt]{amsart}

\usepackage{amssymb}

\usepackage[colorlinks=true, urlcolor=rltblue, pdftex] {hyperref}
\usepackage{color}
\definecolor{rltblue}{rgb}{0,0,0.75}

\usepackage{natbib}

\newtheorem{thm}{Theorem}[section]
\newtheorem{theorem}[thm]{Theorem}

\theoremstyle{definition}
\newtheorem{definition}[thm]{Definition}
\newtheorem{notation}[thm]{Notation}

\def\upto{\mathop{\upharpoonright}}
\def\si{\sigma}
\newcommand{\Nat}{\mathbb{N}}
\newcommand{\ml}{Martin-L\"{o}f }
\newcommand{\sz}{$\Sigma^0_1$\ }

\DeclareMathOperator{\NCR}{NCR}

\title{$K$-trivials are $\NCR$}

\author{Antonio Montalb\'an}

\address{Department of Mathematics\\
University of Chicago\\
5734 S. University ave.\\
Chicago, IL 60637, USA}

\email{antonio@math.uchicago.edu}

\author{Theodore A. Slaman}

\address{Department of Mathematics,
University of California, Berkeley
Berkeley, CA 94720-3840 USA
}
\email{slaman@math.berkeley.edu}

\def\om{\omega}

\begin{document}

\maketitle

\section{Introduction}

In \cite{reimann.slaman:2007,reimann.slaman:2008}, Reimann and Slaman
raise the question ``For which infinite binary sequences $X$ do there
exist continuous probability measures $\mu$ such that $X$ is
effectively random relative to $\mu$?''

\subsection{Randomness relative to continuous measures}

We begin by reviewing the basic definitions needed to precisely
formulate this question.

\begin{notation}
  \begin{itemize}
  \item For $\sigma\in2^{<\omega}$, $[\sigma]$ is the basic open subset of
  $2^\omega$ consisting of those $X$'s which extend $\sigma$.
  Similarly, for $W$ a subset of $2^{<\omega}$, let $[W]$ be the open
  set given by the union of the basic open sets $[\sigma]$ such that
  $\sigma\in W$.
\item For $U\subseteq 2^\omega$, $\lambda(U)$ denotes the
  measure of $U$ under the uniform distribution.  Thus,
  $\lambda([\sigma])$ is $1/2^\ell$, where $\ell$ is the length of
  $\sigma$.
  \end{itemize}
\end{notation}

\begin{definition}
  A \textit{representation} $m$ of a probability measure $\mu$ on
  $2^\omega$ provides, for each $\sigma\in 2^{<\omega}$, a sequence of
  intervals with rational endpoints, each interval containing
  $\mu([\sigma])$, and with lengths converging monotonically to 0.
\end{definition}

\begin{definition} Suppose $Z \in 2^\omega$. A \emph{test relative to
    $Z$}, or \emph{$Z$-test}, is a set $W \subseteq \omega\times
  2^{<\omega}$ which is recursively enumerable in $Z$.  For
  $X\in2^\omega$, $X$ \emph{passes} a test $W$ if and only if there is
  an $n$ such that $X\not\in[W_n]$.
\end{definition}

\begin{definition}  Suppose that $m$ represents the measure $\mu$ on
  $2^\omega$ and that $W$ is an $m$-test.
  \begin{itemize}
  \item $W$ is \emph{correct for $\mu$} if and only if for all
    $n$, $\sum_{\sigma \in W_n} \mu([\sigma]) \leq 2^{-n}.$
  \item $W$ is \emph{Solovay-correct for $\mu$} if and only if
    $\sum_{n\in\omega}\mu([W_n])<\infty$.
  \end{itemize}
\end{definition}

\begin{definition}
  $X\in2^\omega$ is \emph{$1$-random relative to a representation $m$ of
  $\mu$} if and only if $X$ passes every $m$-test which is correct for
  $\mu$.  When $m$ is understood, we say that $X$ is 1-random relative
  to $\mu$.
\end{definition}

By an argument of Solovay, see \cite{nies:2009}, $X$ is $1$-random
relative to a representation $m$ of $\mu$ if an only if for every
$m$-test which is Solovay-correct for $\mu$, there are infinitely
many $n$ such that $X\not\in[W_n]$.  

\begin{definition}
  $X\in\NCR_1$ if and only if there is no representation $m$ of a
  continuous measure $\mu$ such that $X$ is 1-random relative to the
  representation $m$ of $\mu$.
\end{definition}

In \cite{reimann.slaman:2008}, Reimann and Slaman show that if $X$ is
not hyperarithmetic, then there is a continuous measure $\mu$ such
that $X$ is 1-random relative to $\mu$.  Conversely,
Kj{\o}s-Hanssen and Montalb\'an, see \cite{montalban:2005},
%
%
%
have shown that if $X$ is an element of a countable $\Pi^0_1$-class,
then there is no continuous measure for which $X$ is 1-random.  As the
Turing degrees of the elements of countable $\Pi^0_1$-classes are
cofinal in the Turing degrees of the hyperarithmetic sets, the
smallest ideal in the Turing degrees that contains the degrees
represented in $\NCR_1$ is exactly the Turing degrees of the
hyperarithmetic sets.

In \citet{reimann.slaman:ta}, Reimann and Slaman pose the problem to
find a natural $\Pi^1_1$-norm for $\NCR_1$ and to understand its
connection with the natural norm mapping a hyperarithmetic set $X$ to
the ordinal at which $X$ is first constructed.  As of the writing of
this paper, this problem is open in general, but completed in
\cite{reimann.slaman:ta} for $X\in\Delta^0_2$.  

Suppose that $X\in\Delta^0_2$ and that for all $n$,
$X(n)=\lim_{t\to\infty}X_t(n)$, where $X_t(n)$ is a computable
function of $n$ and $t$.  Let $g_X$ be the convergence function for
this approximation, that is for all $n$, $g_X(n)$ is the least $s$
such that for all $t\geq s$ and all $m\leq n$, $X_t(m)=X(m)$.  Let
$f_X$ be function obtained by iterated application of $g_X$:
$f_X(0)=g_X(0)$ and $f_X(n+1)=g_X(f_X(n))$.

For a representation $m$ of a continuous measure $\mu$, the
granularity function $s_m$ maps $n\in\omega$ to the least $\ell$ found
in the representation of $\mu$ by $m$ such that for all $\sigma$ of
length $\ell$, $\mu([\sigma])<1/2^n$.  Note that, $s_m$ is
well-defined by the compactness of $2^\omega$.

\begin{theorem}[Reimann and Slaman~\cite{reimann.slaman:ta}]\label{1.7}
  If $X$ is 1-random relative the representation $m$ of $\mu$, then
  the granularity function $s_m$ for $\mu$ is eventually bounded by
  $f_X$.
\end{theorem}

Thus, there is a continuous measure relative to which $X$ is 1-random
if and only if there is a continuous measure whose granularity is
eventually bounded by $f_X$.  The latter condition is arithmetic,
again by a compactness argument.

\subsection{$K$-triviality}

$K$-triviality is a property of sequences which characterizes another
aspect of their being far from random.  We briefly review this notion and
the results surrounding it.  A full treatment is given in
Nies~\cite{nies:2009}.

For $\sigma\in2^{<\omega}$, let $K(\sigma)$ denote the prefix-free
Kolmogorov complexity of $\sigma$.  Intuitively, given a universal
computable $U$ with domain an antichain in $2^{<\omega}$, $K(\sigma)$
is length of the shortest $\tau$ such that $U(\tau)=\sigma$.
Similarly, for $X\in2^\omega$,  let $K^X(\sigma)$ denote the prefix-free
Kolmogorov complexity of $\sigma$ relative to $X$.  That is, $K^X$ is
determined by a function universal among those computable relative to $X$.

\begin{definition}\label{1.8}
  A sequence $X\in2^\omega$ is \emph{$K$-trivial} if and only if there is a
  constant $k$ such that for every $\ell$, $K(X\restriction\ell) \leq
  K(0^\ell)+k$, where $0^\ell$ is the sequence of $0$'s of length
  $\ell$.
\end{definition}

By early results of Chaitin and Solovay and later results of Nies and
others, there are a variety of equivalents to $K$-triviality and a
variety of properties of the $K$-trivial sets.  For example, $X$ is
$K$-trivial if and and only if for every sequence $R$, $R$ is 1-random
for $\lambda$ if and only if $R$ is 1-random for $\lambda$ relative to
$X$.

In the next section, we will apply the following.

\begin{theorem}[Nies~\cite{nies:2009}, strengthening
  Chaitin~\cite{chaitin:1976}]\label{1.9} 
  If $X$ is $K$-trivial, then there is a computably enumerable and
  $K$-trivial set which computes $X$.
\end{theorem}

The following theorem follows from the work of Nies and others \cite{nies:2009}.
Some versions of  this  property have been used by Ku\v{c}era extensively, e.g.\ in \cite{MR820784}.

\begin{theorem}\label{1.10}  Suppose $X$ is
  $K$-trivial and $\{U_e^X:e\in\om\}$ a uniformly $\Sigma^{0,X}_1$
  family of sets.  Then, there is a computable function $g$ and a
  $\Sigma^0_1$ set $V$ of measure less than 1 such for every $e$, if
  $\lambda(U_e^Z)<2^{-g(e)}$ for every oracle $Z$, then $U_e^X\subseteq V$.
\end{theorem}
\begin{proof} (George Barmpalias)
Let $\big((E_i^e)\big)_{e\in\Nat}$ be a uniform sequence of all oracle \ml
tests. A standard construction of a universal oracle \ml test $(T_i)$ 
(e.g.\ see  \cite{nies:2009}) gives a recursive function $f$ 
such that $\forall Z\subseteq \om\ (E_{f(i,e)}^{e,Z}\subseteq T_i^Z)$ for all $e,i\in\Nat$.
Let $T:=T_2$ and $f(e):=f(2,e)$ for all $e\in\Nat$, so that $\mu( T^Y) \leq 2^{-2}$ for all $Y \in 2^\omega$ and $E_{f(e)}^e\subseteq T$ for all $e\in\Nat$. 
In \cite{MR2336587}  it was shown that $X$ is $K$-trivial iff for some member $T$ of a universal oracle Martin-L\"of test, there is a $\Sigma^0_1$ class $V$ with $T^X \subseteq V$ and $\mu(V) < 1$. 

Now given a uniform enumeration $(U_e)$ of oracle \sz classes we have the following property of $T$:
\begin{quote}\label{Tproperty}
{There is a recursive function $g$ such that for each $e$, \\ either 
$\exists Z\subseteq \om\ (\mu(U_e^{Z})\geq 2^{-g(e)-1})$, or $\forall Z\subseteq \om\ (U_e^Z\subseteq T^Z)$.}
\end{quote}
To see why this is true, note that every $U_e$ can be effectively
mapped to the oracle \ml test $(M_i)$ where $M_i^Z=U_e^Z[s_i]$ and $s_i$
is the largest stage such that $\mu(U_e^Z[s_i])<2^{-i-1}$ (which could be infinity).
Effectively in $e$ we can get an index $n$ of $(M_i)$.
It follows that if $\mu(U^Z_e)<2^{-f(n)-1}$ for all $Z$, then $U_e^X = M_{f(n)}^X = E^{n,X}_{f(n)} \subseteq T^X \subseteq V$. 
So $g(e)=f(n)+1$ is as wanted.
\end{proof}

\subsection{\texorpdfstring{$X$ is $K$-trivial implies $X\in\NCR_1$}{K-trivial implies NCR-1}}

Intuitively, $X\in\NCR_1$ asserts that $X$ is not effectively random
relative to any continuous measure and $X$ is $K$-trivial asserts that
relativizing to $X$ does change the evaluation of randomness relative
to the uniform distribution.  In the next section, we connect the two
notions by showing that if $X$ is $K$-trivial then $X\in\NCR_1$.

\section{The Main Theorem}

\begin{theorem}\label{2.1}
Every $K$-trivial set belongs to $\NCR_1$.
\end{theorem}

\begin{proof}
  Let $Y$ be $K$-trivial and let $\mu$ be a continuous measure with
  representation $m$; we want to show $Y$ is not $\mu$-random.  By
  Theorem~\ref{1.9}, let $X$ be a computably enumerable $K$-trivial
  sequence that computes $Y$.  Let $f$ be the iterated convergence
  function as defined above for the computable approximation to $Y$
  given by approximating $X$'s computation of $Y$.  Since $X$ is
  computably enumerable, $X$ can compute the convergence function for
  its own enumeration and hence $f$ is computable from $X$.

  Let $s_m$ be the granularity function for $\mu$ as represented by
  $m$.  By Theorem~\ref{1.7}, $f$ eventually dominates $s_m$.  By
  changing finitely many values of $f$, we may assume that $f$
  dominates $s_m$ everywhere.  So, we have that for every $n$
  \[
  \mu([Y\upto f(n)])\leq 2^{-n}.
  \]
  Further, we may assume that $f$ can be obtained as the limit of a
  computable function $f(n,s)$ such that for all $s$, $f(n-1,s)\leq
  f(n,s)\leq f(n,s+1)$.  

  We will build an $m$-test $\{S_i:i\in\om\}$ which is Solovay-correct
  for $\mu$ and which $Y$ does not pass, thereby concluding that $Y$
  is not $\mu$-random.  That is, we plan to build $\{S_i:i\in\om\}$ to
  be a uniformly $\Sigma^{0,m}_1$ sequence of sets such that
  $\sum_{i\in\omega}\mu(S_i)$ is bounded and such that there are
  co-finitely $i$ for which $Y\in[S_i]$.  Our construction will not be
  uniform. 

$X$'s $K$-triviality is exploited in the form of Theorem~\ref{1.10}.
Let $V$ and $g$ be given by  Theorem~\ref{1.10} where $\{U_e^X:e\in\om\}$ is a listing of all $\Sigma^{0,X}_1$ sets.
We will build an oracle $\Sigma^0_1$ class $U$ along the construction.
We use the recursion theorem to assume that in advance we know an index $e$ such that $U=U_e$.
During the construction we will make sure that for every oracle $Z$, $\lambda(U^Z)<2^{-g(e)}$.
Theorem~\ref{1.10} then implies that $U^X\subseteq V$ where $V$ is a $\Sigma^0_1$ class of measure less than 1.
To simplify our  notation, let $a$ denote $g(e)$.
Furthermore, assume $a$ is large enough so that $\lambda(V)+2^{-a}<1$.



  We use the approximation to $X$ as a computably enumerable set to
  enumerate approximations to initial segments of $Y$ into the sets
  $S_i$; we rely on the $K$-triviality of $X$ to keep the total
  $\mu$-measure of the $S_i$'s bounded.  

  For each $n>a$ we have a requirement $R_n$ whose task is to
  enumerate $Y\upto f(n)$ into $S_n$.  Let $y_{n,s}=Y_s\upto f(n,s)$
  the stage $s$ approximation to $Y\upto f(n)$.  Let $x_{n,s}$ be the
  initial segment of $X_s$ necessary to compute $y_{n,s}$ and
  $f(n,s)$.  So, if $y_{n,s+1}\neq y_{n,s}$, it is because
  $x_{n,s+1}\neq x_{n,s}$.  In this case, $x_{n,s+1}$ is not only
  different than $x_{n,s}$, but also incomparable.  At stage $s$,
  $R_n$ would like to enumerate $y_{n,s}$ into $S_n$, but before doing
  that it will {\em ask for confirmation} using the fact that
  $U^X\subseteq V$.  Since we are constrained to keep $\lambda(U^X)$
  less than or equal to $2^{-a}$, we will restrict $R_n$ to enumerate
  at most $2^{-n}$ measure into $U^X$.  The reason why we need a
  bit of security before enumerating a string in $S_n$ is that we have
  to ensure that $\sum_i\mu(S_i)$ is bounded.  For this purpose, we
  will only enumerate mass into $S_n$ when we see an equivalent mass
  going into $V$. 

  {\bf Action of requirement $R_n$:} 
  \begin{enumerate}
  \item The first time after $R_n$ is initialized, $R_n$ chooses a
    clopen subset of $2^{\om}$, $\si_n$, of $m$-measure $2^{-n}$, that
    is disjoint form $V_s$ and $U_s^{X_s}$.  Note that since $V$ and
    $U^{X_s}$ have measure less than $\lambda(V)+2^{-a}<1$, we can always
    find such a clopen set.  Furthermore we can chose $\si_n$ to be
    different from the $\si_i$ chosen by other requirements $R_i$,
    $i>a$.  We note the value of $\si_n$ might change if $R_n$ is
    initialized.
  \item To {\em confirm} $x_{n,s}$, requirement $R_n$ enumerates $\si_n$
    into $U^{x_{n,s}}$.  Requirement $R_n$ will not be allowed to
    enumerate anything else into $U^{X_s}$ unless $X_s$ changes below
    $x_{n,s}$.  This way $R_n$ is always responsible for at most
    $2^{-n}$ measure enumerated in $U^{X_s}$.
  \item Then, we wait until a stage $t>s$ such that

    \begin{enumerate}
    \item either $x_{n,s}\not\subseteq x_{n,t}$ (as strings),
    \item or $\si_n\subseteq V_t$.
    \end{enumerate}
    
    Observe that if $x_{n,s}$ is actually an initial segment of $X$,
    then we will have $\si_n\subseteq U^X\subseteq V$.  So, we will
    eventually find such a stage $t$.

    \begin{itemize}
    \item In Case 3(a), we start over with $R_n$.  Note that in
      this case $\si_n$ has come out of $U^{X_t}$, and hence $R_n$ is
      responsible for no measure inside $U^{X_t}$ at stage $t$.
    \item In Case 3(b), if $\mu([y_{n,t}])\leq 2^{-n}$, enumerate
      $y_{n,t}$ into $S_n$.  (Recall that we are allowed to use the
      representation of $\mu$ as an oracle when enumerating $S_n$.)
    \end{itemize}
  \end{enumerate}

  Since we only enumerate $y_{n,t}$ of $\mu$-measure less than $2^{-n}$
  when $\si_n$ is enumerated in $V$, we have that
  \[
  \sum_i\mu(S_i) \leq \lambda(V)<1.
  \]
  It is not hard to check that $\lambda(U^X) \leq \sum_{n=a+1}^\infty 2^{-n}
  =2^{-a}$, so we actually have that $U^X\subseteq V$.  Also notice that
  once $x_{n,s}$ is a initial segment of $X$, we will eventually
  enumerate $\si_n$ into $V$ and an initial segment of $Y$ into $S_n$.
\end{proof}

\bibliographystyle{alpha}

\end{document}